\documentclass[12pt]{article}
\usepackage[utf8]{inputenc}
\usepackage{geometry}

\usepackage[
backend=biber,
doi=false,
url=false,
isbn=false, 
maxbibnames=99,
eprint=false
]{biblatex}
\AtEveryBibitem{\ifentrytype{book}{\clearfield{pages}}{}}

\usepackage{graphicx}
\graphicspath{ {./images/} }
\usepackage{amsmath} 
\usepackage{amssymb} 
\usepackage{listings}
\usepackage[dvipsnames]{xcolor}
\usepackage{color}
\usepackage[vcentermath]{youngtab}
\usepackage{amsthm}
\usepackage{tikz}
\usepackage{lscape}
\usepackage{amsthm}
\usepackage{hyperref}
\usepackage{authblk}
\usepackage[linesnumbered,lined,commentsnumbered,ruled]{algorithm2e}

\usepackage{my_preamble}

\hypersetup{
    colorlinks=true,
    linkcolor=magenta,
    filecolor=magenta,      
    urlcolor=magenta,
    citecolor=magenta
}

\theoremstyle{definition}
\newtheorem{thm}{Theorem}
\newtheorem{theorem}[thm]{Theorem}

\newtheorem{definition}[thm]{Definition}

\newtheorem{lemma}[thm]{Lemma}
\newtheorem{cor}[thm]{Corollary}
\newtheorem{prop}[thm]{Proposition}

\newtheorem{example}[thm]{Example}
\newtheorem{remark}[thm]{Remark}

\title{Logarithmic Voronoi cells}
\author[1]{Yulia Alexandr}
\author[2,3]{Alexander Heaton}
\affil[1]{University of California, Berkeley}
\affil[2]{Max Planck Institute for Mathematics in the Sciences, Leipzig}
\affil[3]{Technische Universit\"at Berlin}
\date{}

\begin{document}

\maketitle

\begin{abstract}
\noindent We study Voronoi cells in the statistical setting  by considering preimages of the maximum likelihood estimator that tessellate an open probability simplex. In general, logarithmic Voronoi cells are convex sets. However, for certain algebraic models, namely finite models, models with ML degree 1, linear models, and log-linear (or toric) models, we show that logarithmic Voronoi cells are polytopes. As a corollary, the algebraic moment map has polytopes for both its fibres and its image, when restricted to the simplex.  We also compute non-polytopal logarithmic Voronoi cells using numerical algebraic geometry. Finally, we determine logarithmic Voronoi polytopes for the finite model consisting of all empirical distributions of a fixed sample size. These polytopes are dual to the logarithmic root polytopes of Lie type A, and we characterize their~faces.

\end{abstract}

\section{Introduction}\label{section:introduction}

For any subset $X \subset \mathbb{R}^n$, the \textit{Voronoi cell} of a point $p \in X$ consists of all points of $\mathbb{R}^n$ which are closer to $p$ than to any other point of $X$ in the Euclidean metric. In this article we discuss the analogous \textit{logarithmic Voronoi cells} which find application in statistics. A discrete statistical model is a subset of the probability simplex $\mathcal{M} \subset \Delta_{n-1}$, since probabilities are positive and sum to $1$. The maximum likelihood estimator $\Phi$ (MLE) sends an empirical distribution $u \in \Delta_{n-1}$ of observed data to the point in the model which best explains the data. This means $p = \Phi(u)$ maximizes the log-likelihood function $\ell_u(p) := \sum_{i=1}^n u_i \log(p_i)$ restricted to $\mathcal{M}$. Note that $\ell_u$ is strictly concave on $\Delta_{n-1}$ and takes its maximum value at $u$. Usually, $u \notin \mathcal{M}$, and we must find the point $\Phi(u) \in \mathcal{M}$ which is closest in the log-likelihood sense. For $p \in \mathcal{M}$ we define the \textit{logarithmic Voronoi cell}
\begin{equation*}
    \log \text{Vor}_\mathcal{M}(p) = \left\{ u \in \Delta_{n-1} : \Phi(u) = p \right\}.
\end{equation*}
Information Geometry \cite{AyJostLeSchwachhofer2017InformationGeometryTEXT} considers MLE in the context of the Kullbach-Leibler divergence of probability distributions, sending data to the nearest point with respect to a Riemannian metric on $\Delta_{n-1}$. Algebraic Statistics \cite{DrtonSturmfelsSullivant2009LecturesOnAlgebraicStatistics} considers the case where $\mathcal{M}$ can be described as either the image or kernel of algebraic maps. Recent work in Metric Algebraic Geometry  \cite{VoronoiCifuentesRanestadSturmfelsWeinstein2018, DiRoccoEklundWeinstein2020BottleneckDegreeOfAlgebraicVarieties, SturmfelsEuclideanDistanceDegree2016, HorobetWeinstein2019OffsetHypersurfacesAndPersistentHomologyAlgebraicVarieties} concerns the properties of real algebraic varieties that depend on a distance metric. Logarithmic Voronoi cells are natural objects of interest in all three subjects.

As an example, consider flipping a biased coin three times. There are four possible outcomes, 3 heads (hhh), 2 heads (hht,hth,thh), 1 head (htt,tht,tth), and 0 heads (ttt).  Parametrically, the \textit{twisted cubic} is given by
\begin{equation*}
    t \mapsto p(t) = \big(t^3, 3t^2(1-t), 3t(1-t)^2, (1-t)^3 \big) \in \mathcal{M}.
\end{equation*}
For this model's many lives, see \cite{Little2019ManyLivesTwistedCubic}.  We compute logarithmic Voronoi cells $\log \text{Vor}_\mathcal{M}(p(t))$ with parameter values
\begin{equation*}
    t \in \left\{ \frac{1}{25}, \frac{2}{25}, \dots, \frac{24}{25} \right\}
\end{equation*}
which live inside the simplex $\Delta_{3} \subset \mathbb{R}^4$, and whose orthogonal projections into $3$-space are shown in Figure \ref{figure:twisted-cubic-lognormal-polytopes}. In this case, the logarithmic Voronoi cells are polytopes, and we get both triangles and quadrilaterals, depending on the point $p(t) \in \mathcal{M}$. The fact that these polytopes are equal to the logarithmic Voronoi cells will follow from Theorem \ref{theorem:log-linear-models-imply-voronoi-equals-polytopes} below.
\begin{figure}[!htb]
    \centering
    \includegraphics[width=0.32\textwidth]{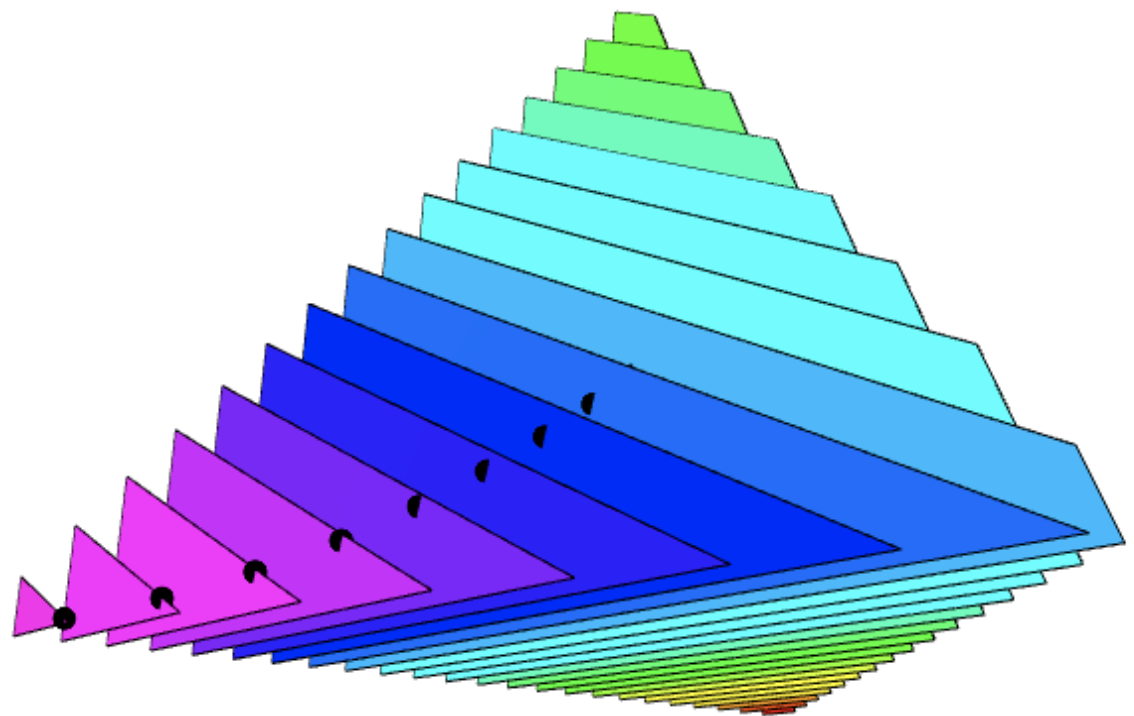} \includegraphics[width=0.32\textwidth]{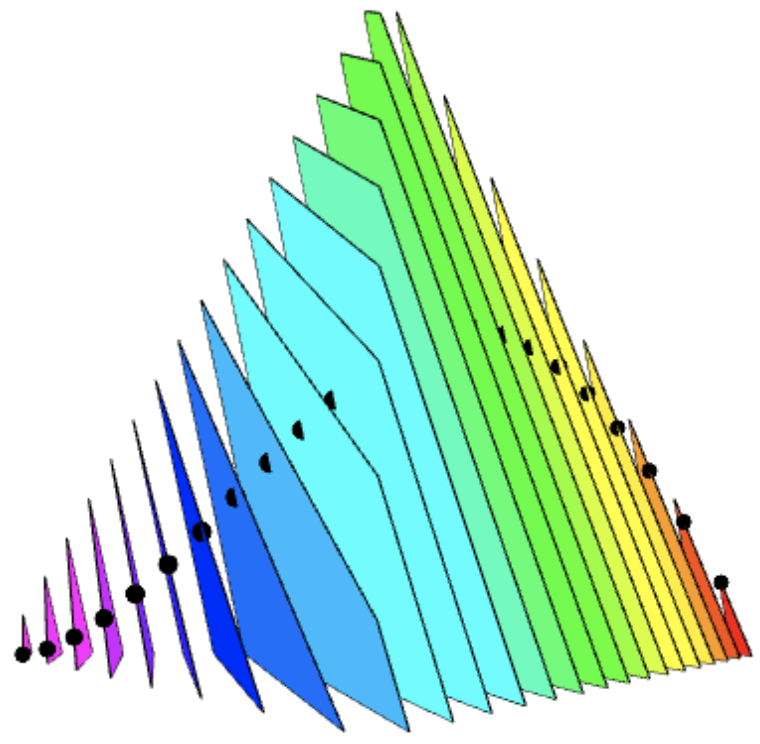} \includegraphics[width=0.32\textwidth]{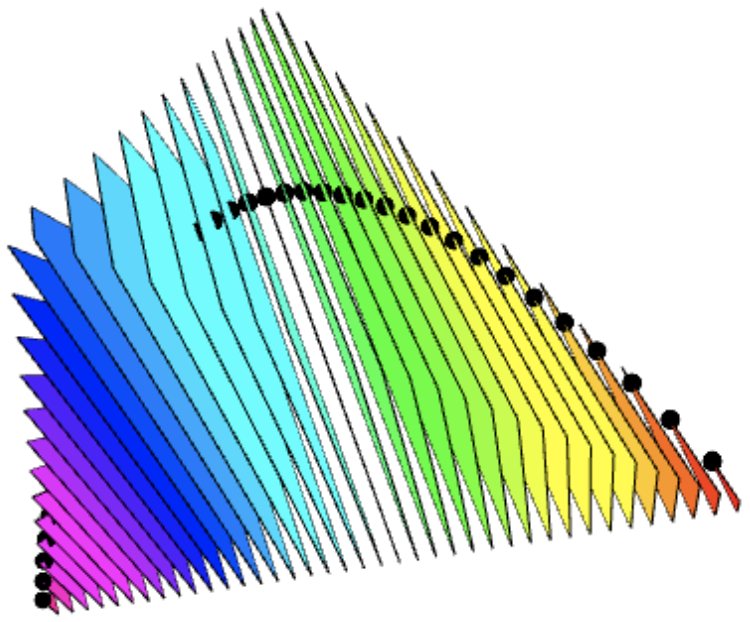}
    \caption{Logarithmic Voronoi cells}
    \label{figure:twisted-cubic-lognormal-polytopes}
\end{figure}

After giving the basic definitions in Section \ref{section:preliminaries}, Section \ref{section:when-are-log-Voronoi-equal-polytopes} describes the relationship between logarithmic Voronoi cells and logarithmic polytopes in the context of algebraic statistics. In particular, we show that ML degree $1$ implies that the logarithmic Voronoi cells are polytopes, and give counterexamples to the converse statement. We also consider both linear models and log-linear (toric) models, showing that both families of statistical models have the property that logarithmic Voronoi cells are polytopes.  These include the twisted cubic of Figure \ref{figure:twisted-cubic-lognormal-polytopes}, decomposable graphical models \cite{Lauritzen1996GraphicalModels}, Bayesian networks \cite{GarciaStillmanSturmfels2005AlgebraicGeometryOfBayesianNetworks}, staged tree models \cite{CollazoGorgenSmith2018ChainEventGraphs, SmithAnderson2008ConditionalIndependenceChainEventGraphs}, multinomial distributions, phylogenetic models, hidden Markov models, and many others arising in applications \cite{PachterSturmfels2005AlgebraicStatisticsForComputationalBiology}. Corollary \ref{corollary:moment-map} states that both the image and fibres of the algebraic moment map are polytopes. In Section \ref{section:numerical-algebraic-geometry} we show how to compute a (not necessarily polytopal) logarithmic Voronoi cell using numerical algebraic geometry. By calculating $\Phi(u)$ for $60,000$ points $u$ with respect to a model of ML degree 39, we demonstrate that logarithmic Voronoi cells can be reliably computed using numerical methods. Finally, in Section \ref{section:finite-point-models} we discuss the historical motivation of Georgy Voronoi and adapt it to the statistical setting by analyzing a model with finitely many points, namely all possible empirical distributions on $n$ states with $d$ trials. We call the polytopes that arise \textit{logarithmic root polytopes} of type $A_{n-1}$, show they are dual to the logarithmic Voronoi cells in Theorem \ref{theorem:log-root-are-dual-log-Voronoi}, and characterize their faces in~Theorem~\ref{theorem:log-root-polytope-face-bijection}.

\section{Preliminaries}\label{section:preliminaries}

We work with the open probability simplex $\Delta_{n-1} \subset \mathbb{R}^n$ defined by
\begin{equation*}
    \Delta_{n-1} := \left\{ u \in \mathbb{R}^n : \sum_{i=1}^n u_i = 1, u_i > 0 \text{ for all } i \in [n] \right\}.
\end{equation*}
A statistical model $\mathcal{M}$ is a subset of the probability simplex. When $\mathcal{M}$ is defined as the intersection of $\Delta_{n-1}$ with an algebraic variety or the image of rational map, we say that $\mathcal{M}$ is an algebraic statistical model \cite{PachterSturmfels2005AlgebraicStatisticsForComputationalBiology, Sullivant2018AlgebraicStatistics}. For any point $u \in \Delta_{n-1}$, the log-likelihood function $\ell_u: \mathbb{R}^n_{>0} \to \mathbb{R}$ is defined by $\ell_u(p) = \sum_{i=1}^n u_i \log(p_i)$. For any model $\mathcal{M} \subset \Delta_{n-1}$, we define the relation $\Phi \subset \Delta_{n-1} \times \mathcal{M}$ by
\begin{equation*}
    (u,p) \in \Phi \iff p \in \text{argmax}_{q \in \mathcal{M}} \left\{ \ell_u(q) : q \in \mathcal{M} \right\}.
\end{equation*}
If $(u,p) \in \Phi$ then we also write $\Phi(u) = p$. We write $\Delta_{n-1}^\mathcal{M}$ for the set of $u \in \Delta_{n-1}$ such that $\Phi(u)$ exists.  Describing the set $\Delta_{n-1}^\mathcal{M}$ and how it extends to the boundary of $\Delta_{n-1}$ is an active area of research, especially with respect to zeros in the data \cite{Fienberg1970Quasi-independenceMaximumLikelihoodEstimationIncompleteContingencyTables, GrossRodriguez2014MaximumLikelihoodPresenceDataZeros}. MLE existence is also connected to polystable and stable orbits in invariant theory \cite{AmendolaKohnReichenbachSeigal2020InvariantTheoryScalingAlgorithmsForMaximumLikelihoodEstimation}. For the important family of log-linear (toric) models, \cite{ErikssonFienbergRinaldoSullivant2006PolyhedralConditionsForNonexistenceOfMLE} shows that positive data $u \in \Delta_{n-1}$ guarantees existence, and in general the MLE exists exactly when the observed margins belong to the relative interior of a certain polytope. See also \cite[Theorem 8.2.1]{Sullivant2018AlgebraicStatistics}. 

Finally, we note that for models with more complicated geometry, $\Phi(u)$ cannot always be computed by finding critical points of $\ell_u$ restricted to manifold points of $\mathcal{M}$. The present article takes the first step of computing logarithmic Voronoi cells for models where critical points of $\ell_u$ succeed in finding the MLE, as well as some interesting finite models. We state necessary assumptions where required. More complicated examples outside the scope of the present article include models of nonnegative rank $r$ matrices, which were studied in \cite{KubjasRobevaSturmfels2015FixedPointsEMAlgorithmNonnegativeRankBoundaries}.

Whenever $p \in \mathcal{M} \subset \mathbb{R}^n$ admits a tangent space at the point $p$, we denote by $N_p \mathcal{M}$ its orthogonal complement with respect to the Euclidean inner product on $\mathbb{R}^n$. We are also interested in the \textit{log-normal space} at the point $p \in \mathcal{M}$, defined by
\begin{equation*}
    \log N_p \mathcal{M} := \left\{ u \in \mathbb{R}^n : \nabla \ell_u(p) \in N_p \mathcal{M} \right\}.
\end{equation*}
Here, $\nabla \ell_u(p)$ is the vector whose entries are given by the partial derivatives of $\ell_u$ with respect to each of the variables $p_1,\dots,p_n$. For an algebraic statistical model $\mathcal{M}$, the \textit{ML degree} is the number of complex critical points of $\ell_u$ on $\mathcal{M}$ for generic data $u \in \Delta_{n-1}$ \cite[p. 140]{Sullivant2018AlgebraicStatistics}.

\begin{lemma}\label{lemma:log-normal-space-linear-subspace}
The log-normal space $\log N_p \mathcal{M}$ is a linear subspace of $\mathbb{R}^n$.
\end{lemma}

\begin{proof}
The normal space $N_p \mathcal{M}$ is a linear subspace. Arrange a basis as the rows of a matrix. Adjoin another row with entries $u_i/p_i$, the partial derivatives of $\ell_u(p)$ with respect to each $p_i$. The maximal minors of the resulting matrix are linear equations in the variables $u_i$ and therefore cut out a linear space of such $u \in \mathbb{R}^n$. This space is the log-normal space at $p$.
\end{proof}

By Lemma \ref{lemma:log-normal-space-linear-subspace}, the intersection of the log-normal space at a point $p\in\M$ with the closed probability simplex $\overline{\Delta_{n-1}}$ is a polytope $\log \text{Poly}_\mathcal{M}(p)$, which we call its \textit{log-normal polytope}. In what follows, when we say that a logarithmic Voronoi cell equals its log-normal polytope, we mean that they are equal as sets, excepting the points in the boundary of the simplex.

\section{Logarithmic Voronoi cells and polytopes}\label{section:when-are-log-Voronoi-equal-polytopes}

\begin{prop}\label{prop:finite-models-logVor-are-polytopes}
Let $\mathcal{M}$ be any finite statistical model. Then the logarithmic Voronoi cells $\log \text{Vor}_\mathcal{M}(p)$ are polytopes for each $p \in \mathcal{M}$.
\end{prop}

\begin{proof}
Fix $p \in \mathcal{M}$. The set of all points $u \in \Delta_{n-1}$ such that $\ell_u(p) \geq \ell_u(q)$ for all $q \in \mathcal{M}$ is the logarithmic Voronoi cell of $p$. Consider some $q \neq p$ but $q \in \mathcal{M}$. Then $\ell_u(p) \geq \ell_u(q)$ becomes the condition that
\begin{equation*}
    \sum_{i=1}^n u_i \log\left( \frac{p_i}{q_i} \right) \geq 0.
\end{equation*}
But this is linear in $u$ and so defines a closed halfspace. Since there are finitely many points in $\mathcal{M}$, we see that the logarithmic Voronoi cell is an intersection of finitely many closed halfspaces (including those defining $\Delta_{n-1}$). Therefore it is a polytope.
\end{proof}

For infinite models, the logarithmic Voronoi cells are, in general, not polytopes. However, if the model is smooth at $p$, the logarithmic Voronoi cell will be contained in the log-normal polytope. Figure \ref{fig:large-example} shows a logarithmic Voronoi cell for $p \in \mathcal{M} \subset \Delta_5 \subset \mathbb{R}^6$ which is not a polytope, but is contained in a polytope. In this case it is the hexagon given by $\log \text{Poly}(p) = \log N_p \mathcal{M} \cap \overline{\Delta_5}$. Since the log-normal space is $2$-dimensional, by choosing an orthonormal basis agreeing with this subspace we can visualize the logarithmic Voronoi cell, despite it living in $\mathbb{R}^6$. We discuss this example in detail in Section \ref{section:numerical-algebraic-geometry}. For more on finite models, see Section \ref{section:finite-point-models}.

\begin{figure}[!htb]
    \centering
    \includegraphics[width=0.6\textwidth]{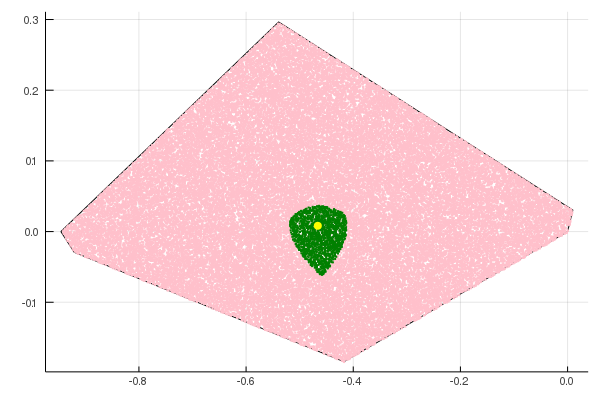}
    \caption{Logarithmic Voronoi cell computed using numerical algebraic geometry}
    \label{fig:large-example}
\end{figure}

\begin{lemma}\label{lemma:smooth-implies-gradient-in-normal-space}
Let $\Phi(u) = p$ for some $p \in \mathcal{M} \subset \Delta_{n-1}$ such that $U \cap \mathcal{M}$ is a manifold for some $p$-neighborhood $U$ in $\mathbb{R}^n$. Then $u$ lies in the logarithmic normal space $\log N_p \mathcal{M}$ and
\begin{equation*}
    \log \text{Vor}_\mathcal{M}(p) \subset \log \text{Poly}_\mathcal{M}(p).
\end{equation*}
\end{lemma}

\begin{proof}
Note that $\ell_u(x):= \sum u_i \log(x_i)$ is a smooth function on any neighborhood of $p \in \mathcal{M}$ contained in $\Delta_{n-1}$. Consider the gradient $\nabla \ell_u(p)$. $\mathbb{R}^n = T_p \mathcal{M} \oplus N_p \mathcal{M}$ and if $\nabla \ell_u(p)$ had any nonzero tangential component then there would exist some $q \in \mathcal{M}$ such that $\ell_u(q) > \ell_u(p)$, contradicting the fact that $\Phi(u) = p$.
\end{proof}

\begin{prop}\label{prop:log-voronoi-are-convex-sets}
Logarithmic Voronoi cells are convex sets.
\end{prop}
\begin{proof} As in the proof of Proposition \ref{prop:finite-models-logVor-are-polytopes}, the logarithmic Voronoi cell of $p$ is defined by the inequalities $\sum_{i\in [n]} u_i \log(p_i/q_i) \geq 0$ for every $q\in \M$, each linear in $u$. Hence, the logarithmic Voronoi cell of $p$ is an intersection of (possibly infinitely many) closed half-spaces, and the result follows. 
\end{proof}

The following theorem concerns algebraic models with ML degree $1$. These were characterized in \cite{Huh2014VarietiesWithMLDegreeOne} and studied further in \cite{DuarteMariglianoSturmfels2019DiscreteStatisticalModelsRationalMaximumLikelihoodEstimator}. They include, for example, Bayesian networks and decomposable graphical models. 

\begin{theorem}\label{theorem:ML-degree-1-implies-log-Voronoi-equal-log-polytope}
Let $\mathcal{M}$ be any algebraic model with ML degree $1$ which is smooth on $\Delta_{n-1}$. Then the logarithmic Voronoi cell at every $p\in \M$ equals its log-normal polytope on $\Delta_{n-1}^\mathcal{M}$.
\end{theorem}

\begin{proof}
We will show that $\log \text{Vor}_\mathcal{M}(p) = \log N_p \mathcal{M} \cap \Delta_{n-1}^\mathcal{M}$. Let $u \in \Delta_{n-1}$ be an element of $\log \text{Vor}_\mathcal{M}(p)$. Then $\Phi(u)=p$ and since $\mathcal{M}$ is smooth, $u \in \log N_p \mathcal{M} \cap \Delta_{n-1}^\mathcal{M}$ by Lemma \ref{lemma:smooth-implies-gradient-in-normal-space}.  For the reverse direction, let $u \in \log N_p \mathcal{M} \cap \Delta_{n-1}^\mathcal{M}$. Recall that $\Phi(u)$ is the argmax of $\ell_u(q)$ over all points $q \in \mathcal{M}$. Since $\Phi(u)$ exists and $\mathcal{M}$ is smooth, this argmax must be among the critical points of $\ell_u$ restricted to $\mathcal{M}$, which include $p$. But since the ML degree is $1$, there is only one complex critical point, and hence $\Phi(u) = p$. Therefore $u$ is in the logarithmic Voronoi cell of $p$, and the result follows.
\end{proof}

\begin{example}\label{example:two-bits-independence-example}

Consider $\mathcal{M} = V(f)$ for $f:\mathbb{C}^4 \to \mathbb{C}^2$ given by the polynomial system
\begin{equation*}
    f(x) = \begin{bmatrix} x_1 x_4 - x_2 x_3 \\ x_1 + x_2 + x_3 + x_4 - 1 \end{bmatrix}: \mathbb{C}^4 \to \mathbb{C}^2
\end{equation*}
A parametrization of this model is given by 
\begin{equation*}
    (p_1, p_2) \mapsto \left( p_1 p_2, \,\, p_1(1-p_2), \,\, (1-p_1)p_2, \,\, (1-p_1)(1-p_2) \right).
\end{equation*}
This is the \textit{independence model} on two binary random variables, and also the Segre embedding of $\mathbb{P}^1 \times \mathbb{P}^1$. The points of this $2$-dimensional model live in the 3-dimensional hyperplane $\sum x_i = 1$ inside $\mathbb{R}^4$, so we can choose a basis agreeing with this hyperplane to plot them.

For each $x\in\M$, we construct an $(m+1) \times n$ matrix $A(x)$ by augmenting the row $\nabla \ell_u$ to the Jacobian matrix $df$:
\begin{equation*}
    A(x)=\left[\begin{array}{rrrr}
x_{4} & -x_{3} & -x_{2} & x_{1} \\
1 & 1 & 1 & 1 \\ u_1/x_1 & u_2/x_2 & u_3/x_3 & u_4/x_4
\end{array}\right].
\end{equation*}
Since our model has codimension two, the $3 \times 3$ minors of $A(x)$ give linear equations describing the log-normal space.
\begin{equation*}
    \begin{array}{c}
        u_{2} - u_{3} - \frac{u_{1} x_{2}}{x_{1}} + \frac{u_{1} x_{3}}{x_{1}} + \frac{u_{2} x_{4}}{x_{2}} - \frac{u_{3} x_{4}}{x_{3}}\\
        u_{1} - u_{4} - \frac{u_{2} x_{1}}{x_{2}} + \frac{u_{1} x_{3}}{x_{1}} - \frac{u_{4} x_{3}}{x_{4}} + \frac{u_{2} x_{4}}{x_{2}}\\
        u_{1} - u_{4} + \frac{u_{1} x_{2}}{x_{1}} - \frac{u_{3} x_{1}}{x_{3}} - \frac{u_{4} x_{2}}{x_{4}} + \frac{u_{3} x_{4}}{x_{3}}\\
        u_{2} - u_{3} + \frac{u_{2} x_{1}}{x_{2}} - \frac{u_{3} x_{1}}{x_{3}} - \frac{u_{4} x_{2}}{x_{4}} + \frac{u_{4} x_{3}}{x_{4}}.
    \end{array}
\end{equation*}
Restricting this space to its intersection with the simplex $u_1 + u_2 + u_3 + u_4 - 1 = 0$ to compute the log-normal polytope, we find that the polytopes are line segments. We plot them for various points on the model in Figure \ref{figure:2by2-independence-lognormal-polytopes}. Since $\M$ has ML degree 1, Theorem \ref{theorem:ML-degree-1-implies-log-Voronoi-equal-log-polytope} tells us that log-Voronoi cells equal log-normal polytopes, so they are also line segments.
\begin{figure}[!htb]
    \centering
    \includegraphics[width=0.32\textwidth]{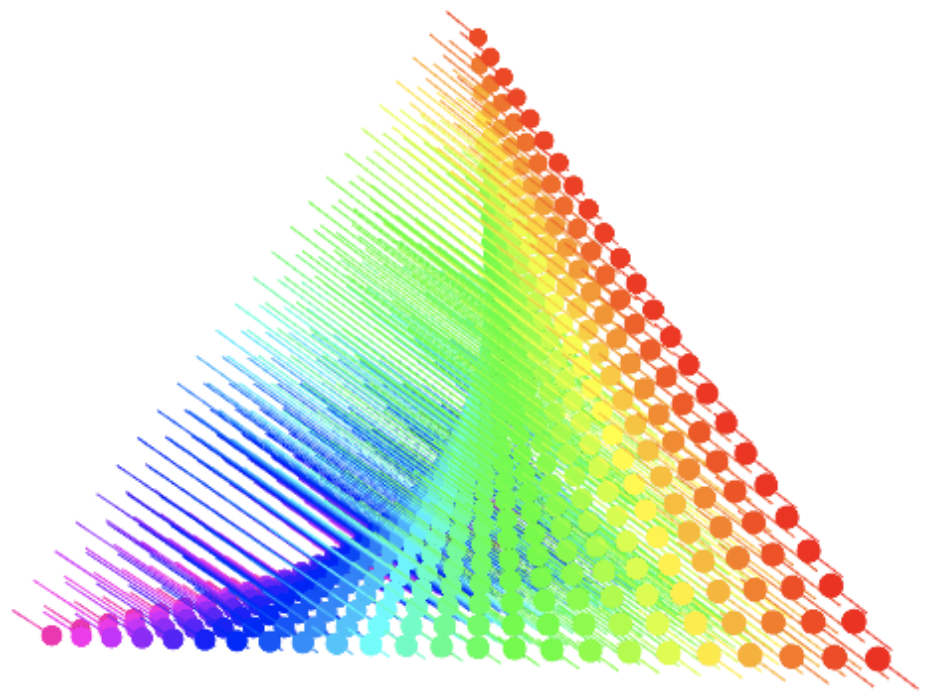} \includegraphics[width=0.32\textwidth]{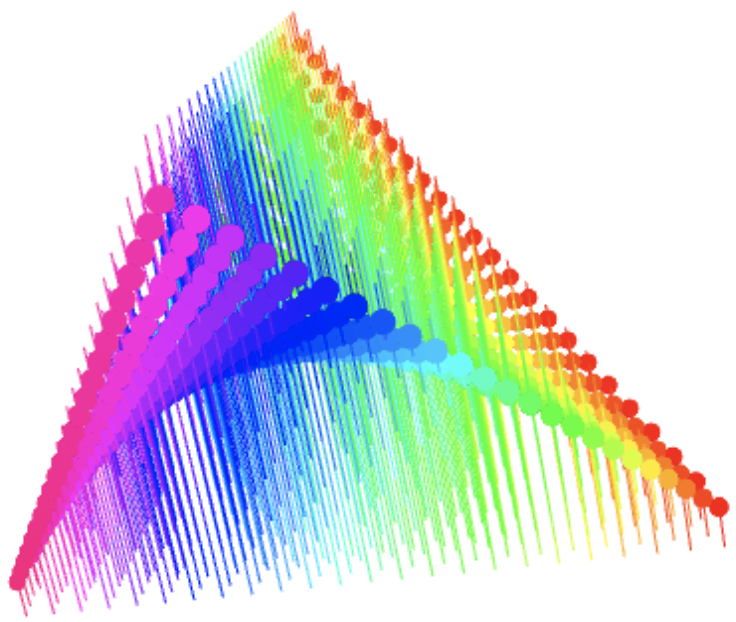} \includegraphics[width=0.32\textwidth]{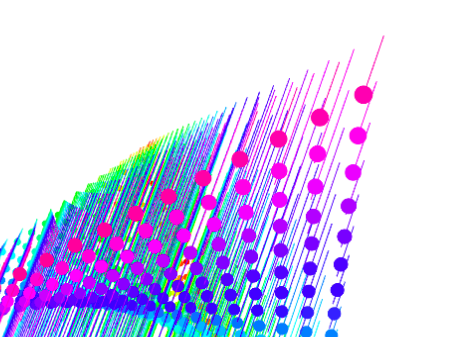}
    \caption{One-dimensional log-normal polytopes at various points}
    \label{figure:2by2-independence-lognormal-polytopes}
\end{figure}
\end{example}

The following Theorem \ref{theorem:cousin-Hardy-Weinberg} shows that the ML degree 1 condition in Theorem \ref{theorem:ML-degree-1-implies-log-Voronoi-equal-log-polytope} is sufficient, but not necessary for the equality of logarithmic Voronoi cells and interiors of respective log-normal polytopes. First consider the independence model of two identically distributed binary random variables. The natural parametrization in a statistical context leads to the Hardy-Weinberg curve defined by $x_2^2 - 4x_1x_3$, which has ML degree $1$ \cite{HuhSturmfels2014LikelihoodGeometry}. A similar-looking model, which has been called the cousin of the Hardy-Weinberg curve \cite{HostenKhetanSturmfels2005SolvingLikelihoodEquations}, is defined by the polynomial $f = x_2^2 - x_1 x_3$. In this case $n=3$ and $\mathcal{M} \subset \Delta_2 \subset \mathbb{R}^3$. It turns out that the ML degree of this model is $2$ \cite[p. 394]{HostenKhetanSturmfels2005SolvingLikelihoodEquations}.

\begin{theorem}\label{theorem:cousin-Hardy-Weinberg}
The algebraic model defined by the polynomial $f = x_2^2 - x_1 x_3$ has ML degree $2$, yet the logarithmic Voronoi cells are equal to their log-normal polytopes.
\end{theorem}

\begin{proof}
Calculate the Jacobian matrix of Lemma \ref{lemma:log-normal-space-linear-subspace} by taking the gradients of $f = x_2^2 - x_1 x_3$ and $g = x_1 + x_2 + x_3 -1$, augmenting this matrix with an additional row of the $u_i/x_i$. Consider the equation of the plane given by the determinant of this matrix. Note that $\M$ is a curve in $\Delta_2$, so the log-normal space at each point is defined by the vanishing of the determinant at that point. This plane has normal vector given by
\begin{equation*}
    \left( \begin{array}{c}
        2 \, x_{1} x_{2}^{2} - x_{1}^{2} x_{3} + x_{1} x_{2} x_{3} + x_{1} x_{3}^{2} \\
        -2 \, x_{1} x_{2}^{2} - 2 \, x_{1} x_{2} x_{3} - 2 \, x_{2}^{2} x_{3} \\
        x_{1}^{2} x_{3} + x_{1} x_{2} x_{3} + 2 \, x_{2}^{2} x_{3} - x_{1} x_{3}^{2}
    \end{array} \right)
\end{equation*}
where $(x_1,x_2,x_3)$ is any point in the common zero set of $f$ and $g$. Consider the cross-product of this vector with the all ones vector, which will give us the direction vector of the log-normal polytope at $(x_1,x_2,x_3)$. Computing and simplifying each coordinate in the quotient ring
\begin{equation*}
    \mathbb{Q}[x_{1}, x_{2}, x_{3}]/\left(x_{1} + x_{2} + x_{3} - 1, -x_{2}^{2} + x_{1} x_{3}\right)\mathbb{Q}[x_{1}, x_{2}, x_{3}],
\end{equation*}
we find that this cross product is given by
\begin{equation*}
\left( \begin{array}{c}
-{\left(x_{2} + x_{3} - 1\right)} x_{3}\\
2 \, {\left(x_{2} + x_{3} - 1\right)} x_{3}\\
-{\left(x_{2} + x_{3} - 1\right)} x_{3}
\end{array} \right)=x_1x_3\left( \begin{array}{c}
1\\
-2\\
1
\end{array} \right).
\end{equation*}

This means that regardless of the point on the curve, the log-normal polytopes will be line segments whose direction vector is $(-1,2,-1)$. We claim that for any distinct $p,q \in \mathcal{M}$ the corresponding line segments are disjoint. Consider the tangent space at some point $x$ in the intersection of $\Delta_2$ and the common zero set of $f$ and $g$. Applying Gaussian elimination to the $2\times 3$ Jacobian matrix, it can be shown that if $2x_2 + x_3 \neq 0$ then all tangent vectors are multiples of
\begin{equation}\label{equation:tangent-vector}
    \left( \frac{x_3 - x_1}{2x_2 + x_3} - 1, \frac{x_1-x_3}{2x_2+x_3}, 1 \right),
\end{equation}
while if $2x_2 + x_3 = 0$ then all tangent vectors are multiples of $(-1,1,0)$. In neither case is it possible that a tangent vector is parallel to $(1,-2,1)$. For $(-1,1,0)$ this is obvious, but for (\ref{equation:tangent-vector}), a contradiction can be derived by showing that if the vector is parallel to $(1,-2,1)$ the first and the last coordinates in (\ref{equation:tangent-vector}) are equal, forcing $x_1 + 4x_2 + x_3 = 0$. But on $\Delta_2$ all coordinates are positive. Thus no line parallel to $(1,-2,1)$ meets the model in two distinct points. We conclude the log-normal polytopes are disjoint, and the result follows from Theorem \ref{theorem:disjoint-implies-polytopes}. 
\end{proof}

\begin{theorem}\label{theorem:disjoint-implies-polytopes}
Let $\mathcal{M}$ be any model smooth on $\Delta_{n-1}$. If all log-normal polytopes for each point $p \in \mathcal{M}$ are disjoint, then the logarithmic Voronoi cells equal log-normal polytopes~on~ $\Delta_{n-1}^\mathcal{M}$.
\end{theorem}

\begin{proof}
We will show that $\log \text{Vor}_\mathcal{M}(p) = \log N_p \mathcal{M} \cap \Delta_{n-1}^\mathcal{M}$. The $\subset$ direction follows from Lemma \ref{lemma:smooth-implies-gradient-in-normal-space}. For the reverse direction, let $u \in \log N_p \mathcal{M} \cap \Delta_{n-1}^\mathcal{M}$. Recall that $\Phi(u)$ is the argmax of $\ell_u(q)$ over all points $q \in \mathcal{M}$. Since $\Phi(u)$ exists and $\mathcal{M}$ is smooth, this argmax must be among the critical points of $\ell_u$ restricted to $\mathcal{M}$, which include $p$. If $\Phi(u)$ were not equal to $p$ then $u$ would be in the intersection of $\Delta_{n-1}$ with the log-normal space to the point $\Phi(u) \in \mathcal{M}$. But the log-normal polytopes were assumed to be disjoint by the hypothesis. Therefore $\Phi(u) = p$, which means that $u\in\log\text{Vor}_\M(p)$, and the result follows.
\end{proof}

Let $f_1(\theta),\cdots,f_r(\theta)$ be nonzero linear polynomials in $\theta$ such that $\sum_{i=1}^rf_i(\theta)=1$. Let $\Theta$ be the set such that $f_i(\theta)>0$ for all $\theta\in \Theta$ and suppose that $\dim\Theta=d$. The model $\M=f(\Theta)\subseteq \Delta_{r-1}$ is called a \textit{discrete linear model} \cite[p.152]{Sullivant2018AlgebraicStatistics}. Linear models appear in \cite[Section 1.2]{PachterSturmfels2005AlgebraicStatisticsForComputationalBiology}. An example is DiaNA's model in Example 1.1 of \cite{PachterSturmfels2005AlgebraicStatisticsForComputationalBiology}.

\begin{theorem}
Let $\mathcal{M}$ be a linear model. Then the logarithmic Voronoi cells are equal to their log-normal polytopes.
\end{theorem}

\begin{proof}
We will show that $\log \text{Vor}_\mathcal{M}(p) = \log N_p \mathcal{M} \cap \Delta_{n-1}$. The $\subset$ direction follows from Lemma \ref{lemma:smooth-implies-gradient-in-normal-space} since an affine linear subspace intersected with $\Delta_{n-1}$ is smooth. For the reverse direction, let $u \in \log N_p \mathcal{M} \cap \Delta_{n-1}$. We must show $\Phi(u) = p$. Since $\ell_u$ is strictly concave on $\Delta_{n-1}$, it is strictly concave when restricted to any convex subset, such as the affine-linear subspace $\mathcal{M}$. Therefore there is only one critical point. Since $\mathcal{M}$ is smooth, $u$ must be in the log-normal space of $\Phi(u)$, and so $\Phi(u)$ must be $p$.
\end{proof}

Next we consider log-linear, or toric, models. These include many important families of statistical models, such as undirected graphical models \cite{GeigerMeekSturmfelsToricAlgebraOfGraphicalModels}, independence models \cite{Sullivant2018AlgebraicStatistics}, and others as mentioned in the introduction. For an $m\times n$ integer matrix $A$ with $\boldsymbol{1}\in\text{rowspan}(A)$, the corresponding \textit{log-linear model} $\M_A$ is defined to be the set of all points $p\in\Delta_{n-1}$ such that $\log(p)\in\text{rowspan}(A)$ \cite[p. 122]{Sullivant2018AlgebraicStatistics}.

\begin{theorem}\label{theorem:log-linear-models-imply-voronoi-equals-polytopes}
Let $A\in \ZZ^{m\times n}$ be an integer matrix such that $\boldsymbol{1}\in\text{rowspan }A$. Let $\M$ be the associated log-linear (toric) model. Then for any point $p\in\M$, the log-Voronoi cell of $p$ is equal to the log-normal polytope at $p$.
\end{theorem}
\begin{proof}
We will show that $\log \text{Vor}_\mathcal{M}(p) = \log N_p \mathcal{M} \cap \Delta_{n-1}$. The forward direction follows from Lemma \ref{lemma:smooth-implies-gradient-in-normal-space}, since these models are smooth off the coordinate hyperplanes (see  \cite[p.150]{Sullivant2018AlgebraicStatistics} and \cite{AmendolaHostenRodriguezSturmfels2019MaximumLikelihoodDegreeToricVarieties}). For the reverse direction, let $u \in \log N_p \mathcal{M}$. Although the log-likelihood function can have many complex critical points, it is strictly concave on log-linear models $\mathcal{M}$ for positive $u$, in particular for $u \in \Delta_{n-1}$. This means that there is exactly one critical point in the positive orthant, and it is the unique solution $p \in \M$ to the linear system $Ap=Au$. \cite[Prop. 2.1.5]{DrtonSturmfelsSullivant2009LecturesOnAlgebraicStatistics}. This is known as \textit{Birch's Theorem}. It follows that $\Phi(u)=p$,~as~desired.
\end{proof}

As a corollary, the polytopes shown in Figure \ref{figure:twisted-cubic-lognormal-polytopes} and Figure \ref{figure:2by2-independence-lognormal-polytopes} are logarithmic Voronoi cells. Following \cite{MichalekSturmfels2019InvitationNonlinearAlgebraTEXT}, define the map sending a point in projective space to a convex combination of the columns $a_i$ of $A$, so that the image is a polytope, namely
\begin{align*}
    \phi_A:\mathbb{P}_\mathbb{C}^{n-1} &\to \mathbb{R}^m\\
    z &\mapsto \frac{1}{\sum_{i=1}^n |z_i|} \sum_{i = 1}^n |z_i| a_i.
\end{align*}
This restricts to what \cite[p.120]{MichalekSturmfels2019InvitationNonlinearAlgebraTEXT} calls the \textit{algebraic moment map} $\phi_A|_{\mathcal{M}_A} = \mu_A:\mathcal{M}_A \to \mathbb{R}^m$, where $\mathcal{M}_A$ is the projective toric variety associated to $A$. The maximum likelihood estimator, then, is the map $\mu_A^{-1} \circ \phi_A$ restricted to $\Delta_{n-1}$, identified as a subset of $\mathbb{P}_\mathbb{C}^{n-1}$ by extending scalars and using the quotient map defining projective space. The fact \cite[Corollary 8.24]{MichalekSturmfels2019InvitationNonlinearAlgebraTEXT} that there is a unique preimage, allowing the definition of $\mu_A^{-1}$, played a crucial role in Theorem \ref{theorem:log-linear-models-imply-voronoi-equals-polytopes}. Thus we have the following
\begin{cor}\label{corollary:moment-map}
For toric models, the logarithmic Voronoi cells are the preimages $\phi_A^{-1}(\mu_A(p))$ intersected with $\Delta_{n-1}$. Thus, $\phi_A|_{\Delta_{n-1}}$ is a map whose image is a polytope and whose fibres are also polytopes.
\end{cor}
For the Segre of Example \ref{example:two-bits-independence-example}, the image is a square and the fibres are line segments, depicted in Figure \ref{fig:toric-fibre-and-image}, which adjoins our Figure \ref{figure:2by2-independence-lognormal-polytopes} with \cite[Figure 2, p.121]{MichalekSturmfels2019InvitationNonlinearAlgebraTEXT}. For more on the algebraic moment map, see~\cite{Sottile2003ToricIdealsRealToricVarietiesMomentMap}. 

\begin{figure}[!htb]
    \centering
    \includegraphics[width=0.9\textwidth]{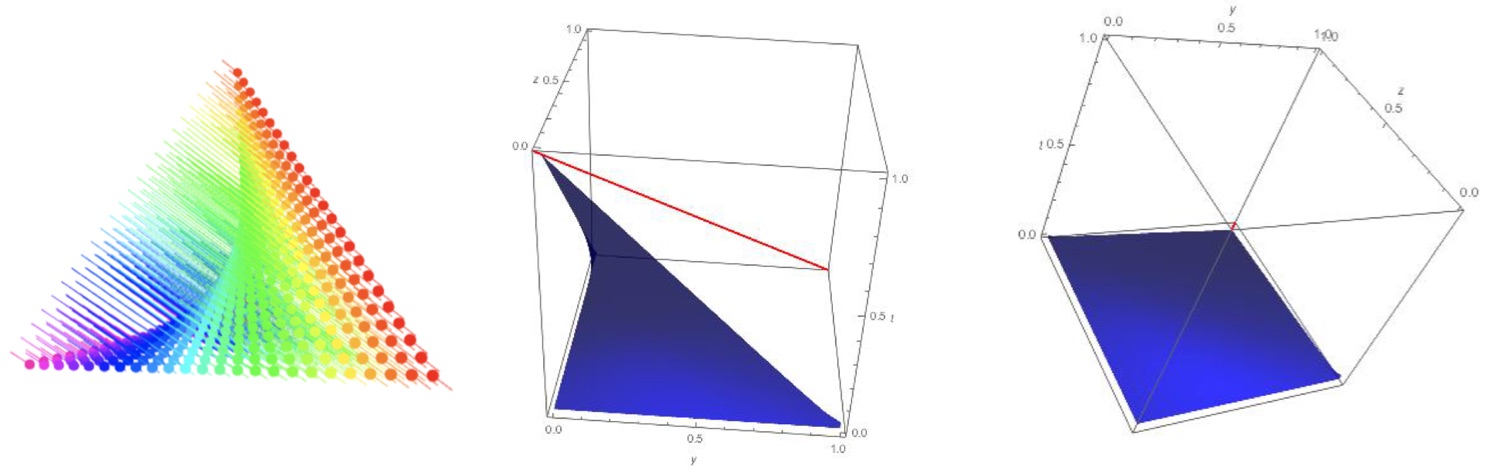}
    \caption{The fibres and image of the moment map for the Segre of Example \ref{example:two-bits-independence-example}}
    \label{fig:toric-fibre-and-image}
\end{figure}

\textbf{Some open questions.} When $\mathcal{M}$ does not equal its log-normal polytope, an interesting open question is how to describe the boundary of the logarithmic Voronoi cells. For Euclidean Voronoi cells of algebraic varieties, this was studied in \cite{VoronoiCifuentesRanestadSturmfelsWeinstein2018}. In particular, are the boundaries algebraic or transcendental? Initial investigations suggest they are transcendental. In addition, when models include singular points, what can we say about the Voronoi cells of the singular locus? This is relevant for the important families of mixture models and secant varieties as in Example \ref{example:segre-mixture-model-numerical-algebraic-geometry}, discussed in Section \ref{section:numerical-algebraic-geometry}. Also, for matrices and tensors of fixed nonnegative rank the geometry is more complicated, and it would be interesting to study logarithmic Voronoi cells in this context, possibly in relation to the basins of attraction of the EM algorithm \cite{KubjasRobevaSturmfels2015FixedPointsEMAlgorithmNonnegativeRankBoundaries}. Finally, we have focused on the discrete case, but continuous distributions could also be investigated. One promising case is linear Gaussian covariance models \cite{Anderson1970CovarianceMatricesLinear}, since their maximum likelihood estimation is an algebraic optimization problem over a spectrahedral cone.

\section{Logarithmic Voronoi cells with numerical algebraic geometry}\label{section:numerical-algebraic-geometry}

An implicit algebraic statistical model $\mathcal{M} \subset \Delta_{n-1}$ is equal to the intersection of $\Delta_{n-1}$ with the zero set of some polynomial map $f:\mathbb{R}^n \to \mathbb{R}^m$, which means that each of the $m$ component functions $f_1,\dots,f_m$ are polynomials in $n$ variables with real coefficients.

\begin{definition}
Let $f$ be the $1 \times m$ row vector whose entries are the polynomials $f_1,\dots,f_m$ in the variables $x_1,\dots,x_n$. We assume that the first polynomial defines the simplex, i.e. $f_1 = \sum_{i=1}^n x_i - 1$.  Let the algebraic set defined by $f_1,\dots,f_m$ have codimension $c$. Let $df$ denote the $m \times n$ Jacobian matrix whose rows are the gradients of $f_1,\dots,f_m$. Let $A$ be a $c \times m$ matrix whose entries are chosen randomly from independent normal distributions. Let $B$ be a similarly chosen random $(m-c) \times (n+c)$ matrix. Let $\left[ \lambda -1 \right]$ be the row vector of length $c+1$ whose first $c$ entries are variables $\lambda_1,\dots,\lambda_c$ and whose last entry is $-1$ and let $I_{n+c}$ be the identity matrix of size $n+c$. We are interested in the following vector equation whose components give $n+c$ polynomial equations in $n+c$ unknowns:
\begin{equation}\label{equation:numerical-algebraic-geometry-formulation}
    \underbrace{\left[ \begin{array}{cc}
        \left[ \lambda -1 \right] \left[ \begin{array}{c}
            A \cdot df \\
            \nabla \ell_u
        \end{array} \right] & f
    \end{array} \right]}_{1 \times (n+m)} \left[ \begin{array}{c}
        I_{n+c} \\
        B
    \end{array} \right] = \underbrace{\left[ 0 \, \cdots \, 0 \right]}_{1 \times (n+c)}.
\end{equation}
\end{definition}

\begin{theorem}\label{theorem:formulation-numerical-algebraic-geometry}
Let $\mathcal{M}$ be the intersection of $\Delta_{n-1}$ and an irreducible algebraic model given by the polynomial map $f:\mathbb{R}^n \to \mathbb{R}^m$. Let $u \in \Delta_{n-1}$ be fixed and generic. With probability $1$, all points $p \in \mathcal{M}$ such that $u \in \log N_p \mathcal{M}$ are among the finitely many isolated solutions to the square system of equations given in (\ref{equation:numerical-algebraic-geometry-formulation}).
\end{theorem}

\begin{proof}
We first refer to \cite[Theorem 1.6]{HuhSturmfels2014LikelihoodGeometry}, which defines the projection map $\text{pr}_2$ and proves that it is generically finite-to-one. As a consequence, if $u \in \Delta_{n-1}$ is generic, then with probability $1$ there will be finitely many critical points of $\ell_u$ restricted to $\mathcal{M}$. If the algebraic set defined by $f$ has codimension $c$ then the dimension of the rowspace of $df$ will be equal to $c$ and the rows will span $N_x \mathcal{M}$ for any generic $x \in \mathcal{M}$ \cite[p.93]{Shafarevich2013BasicAlgebraicGeometry}. With probability $1$, multiplying by the random matrix $A$ will result in a $c \times n$ matrix of full row rank, whose rows also span $N_x \mathcal{M}$. Appending the row $\nabla \ell_u$ and multiplying the resulting matrix by the row vector $\left[ \lambda -1 \right]$ produces $n$ polynomials which evaluate to zero whenever $\nabla \ell_u$ is in the normal space $N_x \mathcal{M}$. Appending the polynomials $f_1,\dots,f_m$ gives a $1 \times (n+m)$ row vector of polynomials evaluating to zero whenever $x \in \mathcal{M}$ and $\nabla \ell_u$ lies in the normal space $N_x \mathcal{M}$. However, this system of equations is overdetermined. Applying Bertini's theorem \cite[Theorem 9.3]{BatesSommeseHauensteinWampler2013NumericallySolvingPolynomialsWithBertini} or \cite[Theorem A.8.7]{SommeseWampler2005NumericalSolutionSystemsPolynomialsTEXT} we can take random linear combinations of these polynomials using $I_{n+c}$ and $B$, and with probability $1$, the isolated solutions of the resulting square system of polynomials will contain all isolated solutions of the original system of equations. The result follows. 
\end{proof}

\begin{remark}
Numerical algebraic geometry \cite{BatesSommeseHauensteinWampler2013NumericallySolvingPolynomialsWithBertini, SommeseWampler2005NumericalSolutionSystemsPolynomialsTEXT} can be used to efficiently find all isolated solutions of a square system of polynomial equations (square means equal number of equations and variables). The system of equations given in Theorem \ref{theorem:formulation-numerical-algebraic-geometry} formulates our problem specifically to take advantage of these tools.
\end{remark}

\begin{remark}
If we are interested in computing the logarithmic Voronoi cell of a specific point $p \in \mathcal{M}$, then we can generate a generic point $u_0 \in \log N_p \mathcal{M}$ by taking a random linear combination of the gradients of $f_1,\dots,f_m$. Using this point $u_0$ we can formulate our system of equations (\ref{equation:numerical-algebraic-geometry-formulation}), one of whose solutions we already know, namely $p$. Using monodromy, we can quickly find many other solutions $p'$ by perturbing our parametrized system of equations through a loop in parameter space. For more details, see \cite{AlexandrHeatonTimme2019ComputingLogarithmicVoronoiCell}. This is especially useful in the case where the ML degree is known a priori, since we can stop our monodromy search after finding ML degree many solutions. This process yields an optimal start system for homotopy continuation, allowing us to almost immediately compute solutions for other data points since we need only follow the ML degree-many solution paths via homotopy continuation. In the next example, we utilize the formulation in Theorem \ref{theorem:formulation-numerical-algebraic-geometry} to numerically compute a logarithmic Voronoi cell in a larger example of statistical interest, \textcolor{black}{a mixture of two binomial distributions, also known as a secant variety.}
\end{remark}

\begin{example}\label{example:segre-mixture-model-numerical-algebraic-geometry}
Bob has three biased coins, one in each pocket, and one in his hand. He flips the coin in his hand, and depending on the outcome, chooses either the coin in his left or right pocket, which he then flips 5 times, recording the total number of heads in the last 5 flips. To estimate the biases of Bob's coins, Alice treats this situation as a $3$-dimensional statistical model $\mathcal{M} \subset \Delta_5 \subset \mathbb{R}^6$. Using implicitization \cite[Section 4.2]{MichalekSturmfels2019InvitationNonlinearAlgebraTEXT}, Alice derives the following algebraic equations describing $\mathcal{M}$:
\begin{equation*}
    f(x) = \left[ \begin{array}{c}
        20x_1x_3x_5-10x_1x_4^2-8x_2^2x_5+4x_2x_3x_4-x_3^3\\
        100x_1x_3x_6-20x_1x_4x_5-40x_2^2x_6+4x_2x_3x_5+2x_2x_4^2-x_3^2x_4\\
        100x_1x_4x_6-40x_1x_5^2-20x_2x_3x_6+4x_2x_4x_5+2x_3^2x_5-x_3x_4^2\\
        20x_2x_4x_6-8x_2x_5^2-10x_3^2x_6+4x_3x_4x_5-x_4^3\\
        x_1+x_2+x_3+x_4+x_5+x_6-1
    \end{array} \right].
\end{equation*}
For a concrete example, consider the point which arises by setting the biases of the coins to $b_1 = \frac{7}{11}, b_2 = \frac{3}{5}, b_3 = \frac{3}{7}$. Explicitly this point $p \in \mathcal{M}$ is
\begin{equation*}
    p = \Big(\frac{518}{9375}, \frac{124}{625}, \frac{192}{625}, \frac{168}{625}, \frac{86}{625}, \frac{307}{9375}\Big).
\end{equation*}
The log-normal space $\log N_p \mathcal{M}$ is $3$-dimensional, becoming a $2$-dimensional polytope when intersected with $\Delta_5 \subset \mathbb{R}^6$. This intersection is the log-normal polytope, in this case, a hexagon. In fact, this hexagon is the (2-dimensional) convex hull of the following six vertices:
\begin{footnotesize}
\begin{align*}
&\left(0, \frac{651}{1625}, 0, \frac{30569}{58500}, \frac{43}{2250}, \frac{3377}{58500}\right)\\
&\left(0, \frac{124}{375}, \frac{88}{375}, \frac{77}{375}, \frac{86}{375}, 0\right)\\
&\left(\frac{8288}{76875}, 0, \frac{3176}{5125}, 0, \frac{1376}{5125}, \frac{307}{76875}\right)\\
&\left(\frac{259}{1875}, 0, \frac{52}{125}, \frac{91}{250}, 0, \frac{307}{3750}\right)\\
&\left(\frac{518}{76875}, \frac{1984}{5125}, 0, \frac{2779}{5125}, 0, \frac{4912}{76875}\right)\\
&\left(\frac{2849}{29250}, \frac{31}{1125}, \frac{8734}{14625}, 0, \frac{903}{3250}, 0\right).
\end{align*}\end{footnotesize}
By choosing an orthonormal basis agreeing with $\log N_p \mathcal{M}$ we can plot this hexagon, though it lives in $\mathbb{R}^6$. Figure \ref{fig:large-example} shows the log-normal polytope and our numerical approximation of the logarithmic Voronoi cell (which is not a polytope) surrounding the point $p$. By rejection sampling, we computed $60000$ points $u_1,u_2,\dots,u_{60000} \in \log N_p \mathcal{M} \cap \Delta_{n-1}$ in the log-normal polytope. By a result in \cite{HostenKhetanSturmfels2005SolvingLikelihoodEquations}, we know that the ML degree of this model is $39$. Using the formulation presented in Theorem \ref{theorem:formulation-numerical-algebraic-geometry}, we successfully computed all $39$ complex critical points for each $\ell_{u_i}, i \in \{1,2,\dots,60000\}$ restricted to $\mathcal{M}$. We easily find each $\Phi(u_i)$ by comparing the $39$ values, choosing the maximum. If $p = \Phi(u_i)$ then $u_i \in \log \text{Vor}_\mathcal{M}(p)$ and we color that point green in Figure \ref{fig:large-example}, while if $p \neq \Phi(u_i)$ we color the point pink. The repeated computations of each set of $39$ critical points were accomplished using the software \texttt{HomotopyContinuation.jl} \cite{BreidingTimme2017HomotopyContinuationJL}, which can efficiently compute the isolated solutions to systems of polynomial equations using homotopy continuation \cite{BatesSommeseHauensteinWampler2013NumericallySolvingPolynomialsWithBertini, SommeseWampler2005NumericalSolutionSystemsPolynomialsTEXT}.  A full description of the \texttt{Julia} code needed to compute this example can be found online at \cite{AlexandrHeatonTimme2019ComputingLogarithmicVoronoiCell}.

\end{example}

\section{All empirical distributions of fixed sample size}\label{section:finite-point-models}

Consider running experiments with sample size $d$ and choosing the model defined by
\begin{equation*}
    \mathcal{M} := \frac{ \mathbb{Z}^n \cap d \cdot \Delta_{n-1} }{d}.
\end{equation*}

Philosophically, $\mathcal{M}$ is the \textit{chaotic universe model}. Adopting this model is to abandon the idea that experiments tell us about some simpler underlying truth, since the experimental data will always lie exactly on the model. In this section we investigate the Euclidean and logarithmic Voronoi cells for $p \in \mathcal{M}_{n,d}$. For convenience we work with the scaled set $d \cdot \Delta_{n-1}$ since all polytopes considered will be combinatorially equivalent to those we could define in $\Delta_{n-1}$. Then we define $\mathcal{M}_{n,d}$ as the $N:= \binom{n+d-1}{d}$ nonnegative integer vectors summing to $d$. Thus $(p_1,p_2,\dots,p_n) = p \in \mathcal{M}_{n,d}$ has all coordinates $p_i \in \mathbb{N}$. These vectors can be used to create a projective toric variety, the $d$th Veronese embedding of $\mathbb{P}^{n-1}$ into $\mathbb{P}^{N-1}$ \cite[Chapter 8]{MichalekSturmfels2019InvitationNonlinearAlgebraTEXT}, but instead we treat them as the model itself. By Proposition \ref{prop:finite-models-logVor-are-polytopes}, the logarithmic Voronoi cells for $p \in \mathcal{M}_{n,d}$ are polytopes. For any $p \in \mathcal{M}_{n,d}$ such that all coordinates $p_i > 1$, we will provide a full characterization of the faces of the corresponding logarithmic root polytopes in Theorem \ref{theorem:log-root-polytope-face-bijection}. Theorem \ref{theorem:log-root-are-dual-log-Voronoi} shows that these logarithmic root polytopes are dual to the logarithmic Voronoi cells. These are the main results of the section. Again using orthogonal projection from $\mathbb{R}^4$, Figure \ref{fig:finite-simplex-tessellation} shows all the logarithmic Voronoi cells for interior points of $\mathcal{M}_{4,9}$ and $\mathcal{M}_{4,10}$.
\begin{figure}[!htb]
    \centering
    \includegraphics[scale=0.5]{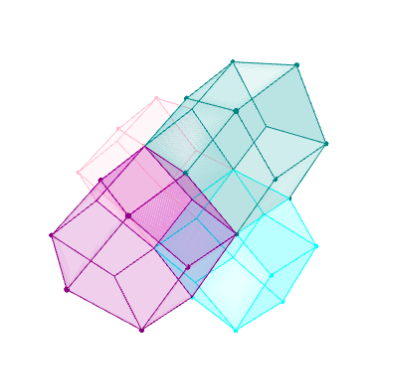}
    \includegraphics[scale=0.5]{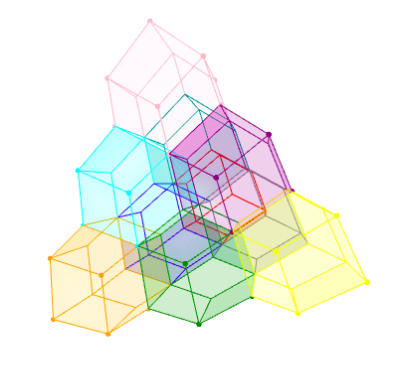}
    \caption{Logarithmic Voronoi cells (rhombic dodecahedra) of interior points for $n=4$, $d=9$ (on the left) and $d=10$ (on the right).}
    \label{fig:finite-simplex-tessellation}
\end{figure}

The Euclidean Voronoi cells for $p \in \mathcal{M}_{n,d}$ are the duals of \textit{root polytopes} of type $A_{n-1}$, i.e. the facets are defined by inequalities whose normal vectors are $\{ e_i - e_j : i \neq j \}$. Root polytope often refers to the convex hull of the origin and the positive roots $\{e_i - e_j : i < j \}$. These were studied in \cite{GelfandGraevPostnikov1997CombinatoricsOfHypergeometricFunctionsAssociatedWithPositiveRoots} in terms of their relationship to certain hypergeometric functions. However, we define root polytopes to be the convex hull of \textit{all} roots, as studied in \cite{Cho1999PolytopesOfRootsOfTypeAn}. We also note that these polytopes are \textit{Young orbit polytopes} for the partition $(n-1,1)$ and find application in combinatorial optimization \cite{Onn1993GeometryComplexityCombinatoricsPermutationPolytopes}.

Denote the $(n-1)$-dimensional root polytope by $P_n \subset \mathbb{R}^n$, so that the Euclidean Voronoi cells of $p \in \mathcal{M}_{n,d}$ are the dual $P_n^*$. The volume of $P_n$ is equal to $\frac{n}{(n-1)!} C_{n-1}$, where $C_{n-1}$ is a Catalan number. Every nontrivial face of $P_n$ is a Cartesian product of two simplices, and corresponds to a pair of nonempty, disjoint subsets $I,J \subset [n]$. Every $m$-dimensional face of $P_n$ is the convex hull of the vectors $\{e_i - e_j : i \in I, j \in J\}$ with $|I|+|J| = m+2$, so there is a bijection between nontrivial faces and the set of ordered partitions of subsets of $[n]$ with two blocks \cite[Theorem 1]{Cho1999PolytopesOfRootsOfTypeAn}. This result is related to the face description of $\Pi_{n-1}$, the permutahedron, since $P_n$ is a \textit{generalized permutahedron} and can be obtained by collapsing certain faces of $\Pi_{n-1}$. 

In the logarithmic setting, analogous polytopes $\log P_n(p)$ exist, playing the same role as the root polytopes in the Euclidean case. However, their details are more complicated. The correct modifications motivate the following definition.

\begin{definition}\label{definition:logarithmic-root-polytopes}
The \textit{logarithmic root polytope} for $p \in \mathcal{M}_{n,d}$ is defined as the convex hull of the $2 \binom{n}{2}$ vertices $v_{ij}$ for $i \neq j \in [n]$ given by the formulas
\begin{equation*}
    v_{ij} := \frac{1}{b_j p_j - a_i p_i} \left[ a_i e_i - b_j e_j - \frac{(a_i - b_j)}{n}\mathbf{1} \right]
\end{equation*}
where
\begin{equation*}
    \begin{array}{ccc}
        a_i := \log(\frac{p_i + 1}{p_i}) & & b_j := \log(\frac{p_j}{p_j - 1})
    \end{array}
\end{equation*}
and where $\mathbf{1}:= \sum_{k \in [n]} e_k$. Note that $a_i, b_j >0$ are always positive real numbers and all vectors $v_{ij}$ are orthogonal to $\mathbf{1}$. We denote the polytope by $\log P_n(p)$.
\end{definition}

The statement and proof of the following Theorem \ref{theorem:log-root-polytope-face-bijection} was inspired by and closely follows \cite[Theorem 1]{Cho1999PolytopesOfRootsOfTypeAn}. However, significant details needed to be modified. For example, the linear functional
\begin{equation*}
    g = (1,0,-1,1,-1,0,-1)
\end{equation*}
is replaced by
\begin{equation*}\scriptsize
    \left[ \begin{array}{c}
        -a_{1} a_{4} b_{3} b_{5} p_{1} - a_{1} a_{4} b_{3} b_{7} p_{1} - a_{1} a_{4} b_{5} b_{7} p_{1} - a_{1} b_{3} b_{5} b_{7} p_{1} + a_{4} b_{3} b_{5} b_{7} p_{3} + a_{4} b_{3} b_{5} b_{7} p_{4} + a_{4} b_{3} b_{5} b_{7} p_{5} + a_{4} b_{3} b_{5} b_{7} p_{7} \\ 0 \\ a_{1} a_{4} b_{5} b_{7} p_{1} - a_{1} a_{4} b_{3} b_{5} p_{3} - a_{1} a_{4} b_{3} b_{7} p_{3} - a_{1} b_{3} b_{5} b_{7} p_{3} - a_{4} b_{3} b_{5} b_{7} p_{3} + a_{1} a_{4} b_{5} b_{7} p_{4} + a_{1} a_{4} b_{5} b_{7} p_{5} + a_{1} a_{4} b_{5} b_{7} p_{7} \\ a_{1} b_{3} b_{5} b_{7} p_{1} + a_{1} b_{3} b_{5} b_{7} p_{3} - a_{1} a_{4} b_{3} b_{5} p_{4} - a_{1} a_{4} b_{3} b_{7} p_{4} - a_{1} a_{4} b_{5} b_{7} p_{4} - a_{4} b_{3} b_{5} b_{7} p_{4} + a_{1} b_{3} b_{5} b_{7} p_{5} + a_{1} b_{3} b_{5} b_{7} p_{7} \\ a_{1} a_{4} b_{3} b_{7} p_{1} + a_{1} a_{4} b_{3} b_{7} p_{3} + a_{1} a_{4} b_{3} b_{7} p_{4} - a_{1} a_{4} b_{3} b_{5} p_{5} - a_{1} a_{4} b_{5} b_{7} p_{5} - a_{1} b_{3} b_{5} b_{7} p_{5} - a_{4} b_{3} b_{5} b_{7} p_{5} + a_{1} a_{4} b_{3} b_{7} p_{7} \\ 0 \\ a_{1} a_{4} b_{3} b_{5} p_{1} + a_{1} a_{4} b_{3} b_{5} p_{3} + a_{1} a_{4} b_{3} b_{5} p_{4} + a_{1} a_{4} b_{3} b_{5} p_{5} - a_{1} a_{4} b_{3} b_{7} p_{7} - a_{1} a_{4} b_{5} b_{7} p_{7} - a_{1} b_{3} b_{5} b_{7} p_{7} - a_{4} b_{3} b_{5} b_{7} p_{7} 
    \end{array} \right].
\end{equation*}
This linear functional plays the same role for the logarithmic root polytope of $(p_1,p_2,\dots,p_7) \in \mathcal{M}_{7,d}$
as $g$ plays for the usual root polytope in the proof of \cite[Theorem 1]{Cho1999PolytopesOfRootsOfTypeAn}.

\begin{theorem}\label{theorem:log-root-polytope-face-bijection}
For $m \in \{0,1,\dots,n-2\}$, every $m$-dimensional face of the logarithmic root polytope for $p \in \mathcal{M}_{n,d}$ is given by the convex hull of the vertices $v_{ij}$ for $i \in I, j \in J$, where $I,J$ are disjoint nonempty subsets of $[n]$ such that $|I| + |J| = m+2$. Thus there is a bijection between nontrivial faces and the set of ordered partitions of subsets of $[n]$ with two blocks, where the dimension of the face corresponding to $(I,J)$ is $|I|+|J|-2$.
\end{theorem}

\begin{proof}
Each face of a polytope can be described as the subset of the polytope maximizing a linear functional. Recall that we have fixed some $p \in \mathcal{M}_{n,d}$ with all $p_k > 1$ and that
\begin{equation*}
    \begin{array}{ccc}
        a_i := \log(\frac{p_i + 1}{p_i}) & \text{ and } & b_j := \log(\frac{p_j}{p_j - 1}).
    \end{array}
\end{equation*}
In our formula (\ref{equation:Yulia-rewrite-coordinate-formulas}) we use a shorthand for writing square-free monomials in the $a_1,a_2,\dots,a_n$ and the $b_1,b_2,\dots,b_n$. For example if $I=\{1,2,4\}$ then $a^I = a_1 a_2 a_4$, while if $J = \{3,5\}$ then $b^J = b_3 b_5$. For a pair of disjoint nonempty subsets $I,J$ of $[n]$ we define the linear functional $g_{IJ} = (g_1,g_2,\dots,g_n) \in (\mathbb{R}^n)^*$ by the formulas
\begin{equation}\label{equation:Yulia-rewrite-coordinate-formulas}
    \begin{array}{cc}
        \text{If } \ell \in I, & g_\ell = \sum_{i \in I \setminus \ell} a^{I \setminus \{\ell,i\}} b^J (a_i p_i - a_\ell p_\ell ) + \sum_{j \in J} a^{I \setminus \ell} b^{J \setminus j} (b_j p_j - a_\ell p_\ell)\\
        \text{If } \ell \in J, & g_\ell = \sum_{i \in I} a^{I \setminus i} b^{J \setminus \ell}(a_i p_i - b_\ell p_\ell ) + \sum_{j \in J \setminus \ell} a^I b^{J \setminus \{\ell,j\}} (b_j p_j - b_\ell p_\ell) \\
        \text{Else, } & g_\ell = 0.
    \end{array}
\end{equation}
Then the convex hull of the vectors $\{ v_{ij} : i \in I, j \in J\}$ is the face maximizing $g_{IJ}$. To see this, first note that $g_{IJ} \cdot \mathbf{1} = 0$. Because of this fact we can ignore the component of $v_{ij}$ in the $\mathbf{1}$ direction. Recall that
\begin{equation*}
    v_{ij} := \frac{1}{b_j p_j - a_i p_i} \left[ a_i e_i - b_j e_j - \frac{(a_i - b_j)}{n}\mathbf{1} \right],
\end{equation*}
so that to evaluate $g_{IJ}$ on $v_{ij}$ it is enough to evaluate on
\begin{equation*}
    \frac{1}{b_j p_j - a_i p_i} \left[ a_i e_i - b_j e_j \right].
\end{equation*}
Recalling that the $a_i$ and $b_j$ are always positive and that the $p_k > 1$, it can be seen that $g_{IJ}$ takes equal value on every vertex $v_{rs}$ for $r \in I, s \in J$, and strictly less on every other vertex.  We omit the details of the admittedly lengthy calculation, but note that the common maximum value attained on all vertices $v_{rs}$ for $r \in I, s \in J$, is equal to
\begin{equation*}
    \sum_{i \in I} a^{I \setminus i} b^J + \sum_{j \in J} a^I b^{J \setminus j}.
\end{equation*} 
Conversely, given an arbitrary linear functional $f = (f_1,f_2,\dots,f_n)$ determining a nontrivial face $F$, collect the indices where its components are nonnegative in a set $I$ and the indices where its components are negative in a set $J$. Then $(I,J)$ is a partition of $[n]$ and we refer to the same formulas (\ref{equation:Yulia-rewrite-coordinate-formulas}) as above in order to define the sets $(I',J')$ as follows. If $I \neq \emptyset $ and $J \neq \emptyset$ then let
\begin{align*}
    I' &:= \{ i : f_i/g_i = \text{max}( f_\ell / g_\ell : \ell \in I )\}\\
    J' &:= \{ j : f_j/g_j = \text{max}( f_\ell / g_\ell : \ell \in J )\}.
\end{align*}
If $I = \emptyset$ then let
\begin{align*}
    I' &:= \{ i : f_i/g_i = \text{min}( f_\ell / g_\ell : \ell \in J )\}\\
    J' &:= \{ j : f_j/g_j = \text{max}( f_\ell / g_\ell : \ell \in J )\},
\end{align*}
while if $J = \emptyset$ then let
\begin{align*}
    I' &:= \{ i : f_i/g_i = \text{max}( f_\ell / g_\ell : \ell \in I )\}\\
    J' &:= \{ j : f_j/g_j = \text{min}( f_\ell / g_\ell : \ell \in I )\}.
\end{align*}
Note that the face $F$ is the convex hull of the vectors $\{ v_{ij} : i \in I', j \in J' \}$ and hence $(I',J')$ are determined independently of the choice of linear functional which maximizes the~given~face.

Now we show that the dimension of the face corresponding to disjoint nonempty sets $I,J$ of $[n]$ is $|I|+|J|-2$. Let $I = \{i_1,\dots,i_{|I|}\}$ and $J = \{ j_1, \dots, j_{|J|}\}$. Then
\begin{equation*}
    X = \{ v_{i_1,j_\ell} : \ell =1,\dots,|J| \} \cup \{ v_{i_\ell,j_1} : \ell =2,\dots,|I| \}
\end{equation*}
is a maximal linearly independent subset of $|I|+|J|-1$ of the vectors $v_{ij}, i \in I, j \in J$. In addition, for any $i \in I, j \in J$ either $v_{ij} \in X$ or we can write it as an affine combination (coefficients sum to 1) of vectors in X, namely
\begin{equation*}
    v_{i,j} = \left(\frac{b_{j_1}p_{j_1} - a_i p_i}{b_j p_j - a_i p_i}\right)v_{i,j_1} - \left(\frac{b_{j_1}p_{j_1} - a_{i_1}p_{i_1}}{b_j p_j - a_i p_i}\right)v_{i_1,j_1} + \left(\frac{b_j p_j - a_{i_1} p_{i_1}}{b_j p_j - a_i p_i}\right)v_{i_1,j}.
\end{equation*}
Hence, $X$ is an affine basis of the face corresponding to $I,J$, whose dimension is $|X|-1$, which is $|I|+|J|-2$ as desired. This completes the proof.
\end{proof}

\begin{example}
Let $n=6$, $I = \{ 1,4\}$, $J = \{2,3,5\}$ and $p = (2,15,3,5,9,6)$. We implemented the formulas (\ref{equation:Yulia-rewrite-coordinate-formulas}) in floating point arithmetic (due to the logarithms) and obtain (shown to only three digits)
\begin{equation*}
    g_{IJ} = (0.00415, -0.00200, -0.00398, 0.00474, -0.00291, -0.000).
\end{equation*}
We can evaluate this linear functional on the vertices $v_{ij}$ for $i \neq j$ where $i,j \in [6]$ and obtain the following values, which are as expected.
\begin{equation*}\scriptsize
    \begin{array}{rl}
        0.008135843945 & v(1, 2) = (1.56, -0.559, -0.251, -0.251, -0.251, -0.251) \\ 0.008135843948 & v(1, 3) = (1.00, 0.000, -1.00, 0.000, 0.000, 0.000) \\ 0.002052114856 & v(1, 4) = (1.23, -0.0997, -0.0997, -0.832, -0.0997, -0.0997) \\ 0.008135843948 & v(1, 5) = (1.43, -0.192, -0.192, -0.192, -0.665, -0.192) \\ 0.005950315119 & v(1, 6) = (1.30, -0.132, -0.132, -0.132, -0.132, -0.776) \\ -0.007192386292 & v(2, 1) = (-1.41, 0.405, 0.250, 0.250, 0.250, 0.250) \\ 0.005982647332 & v(2, 3) = (0.229, 0.488, -1.40, 0.229, 0.229, 0.229) \\ -0.008044880930 & v(2, 4) = (0.179, 0.615, 0.179, -1.33, 0.179, 0.179) \\ 0.002322216671 & v(2, 5) = (0.0963, 0.796, 0.0963, 0.0963, -1.18, 0.0963) \\ -0.001027169161 & v(2, 6) = (0.156, 0.669, 0.156, 0.156, 0.156, -1.29) \\ -0.007691322875 & v(3, 1) = (-1.20, 0.129, 0.679, 0.129, 0.129, 0.129) \\ -0.005863205380 & v(3, 2) = (-0.213, -0.615, 1.47, -0.213, -0.213, -0.213) \\ -0.008723741580 & v(3, 4) = (-0.0426, -0.0426, 1.10, -0.926, -0.0426, -0.0426) \\ -0.004075725208 & v(3, 5) = (-0.144, -0.144, 1.32, -0.144, -0.742, -0.144) \\ -0.004962538041 & v(3, 6) = (-0.0762, -0.0762, 1.17, -0.0762, -0.0762, -0.867) \\ -0.004242519680 & v(4, 1) = (-1.28, 0.179, 0.179, 0.563, 0.179, 0.179) \\ 0.008135843941 & v(4, 2) = (-0.153, -0.713, -0.153, 1.33, -0.153, -0.153) \\ 0.008135843947 & v(4, 3) = (0.122, 0.122, -1.21, 0.720, 0.122, 0.122) \\ 0.008135843947 & v(4, 5) = (-0.0723, -0.0723, -0.0723, 1.15, -0.865, -0.0723) \\ 0.004743470845 & v(4, 6) = (0.000, 0.000, 0.000, 1.00, 0.000, -1.00) \\ -0.007271750954 & v(5, 1) = (-1.36, 0.224, 0.224, 0.224, 0.464, 0.224) \\ -0.001944541355 & v(5, 2) = (-0.0700, -0.866, -0.0700, -0.0700, 1.15, -0.0700) \\ 0.004878535171 & v(5, 3) = (0.186, 0.186, -1.32, 0.186, 0.579, 0.186) \\ -0.008151512920 & v(5, 4) = (0.117, 0.117, 0.117, -1.21, 0.745, 0.117) \\ -0.002105123850 & v(5, 6) = (0.0880, 0.0880, 0.0880, 0.0880, 0.811, -1.16) \\ -0.006239195419 & v(6, 1) = (-1.31, 0.195, 0.195, 0.195, 0.195, 0.528) \\ 0.001256424608 & v(6, 2) = (-0.129, -0.756, -0.129, -0.129, -0.129, 1.27) \\ 0.005540018448 & v(6, 3) = (0.144, 0.144, -1.25, 0.144, 0.144, 0.672) \\ -0.005547164875 & v(6, 4) = (0.0602, 0.0602, 0.0602, -1.11, 0.0602, 0.867) \\ 0.002536892813 & v(6, 5) = (-0.0448, -0.0448, -0.0448, -0.0448, -0.915, 1.09)
    \end{array}
\end{equation*}
\end{example}

\begin{theorem}\label{proposition:logarithmic-root-polytopes}\label{theorem:log-root-are-dual-log-Voronoi}
The logarithmic Voronoi cells for $p \in \mathcal{M}_{n,d}$ with all $p_i > 1$ are the dual polytopes $(\log P_n(p))^*$ of the logarithmic root polytopes $\log P_n(p)$.
\end{theorem}

\begin{proof}
Given a point $p \in \mathcal{M}_{n,d}$, the logarithmic Voronoi cell can be defined as the intersection of $d \cdot \Delta_{n-1}$ with all the halfspaces $H_q(u) \geq 0$ for all points $q \in \mathcal{M}_{n,d}$ with $q \neq p$, where
\begin{equation*}
    H_q(u) := \sum_{i \in [n]} u_i \log\left( \frac{p_i}{q_i} \right).
\end{equation*}
We say that this system of inequalities is \textit{sufficient} to define the logarithmic Voronoi cell. However, not all of these inequalities are necessary. Lemma \ref{lemma:logarithmic-Voronoi-sufficient-inequalities} shows that a certain set of $2\binom{n}{2}$ inequalities is sufficient for all $n \in \mathbb{Z}_{\geq 2}$. These are the inequalities $H_q(u) \geq 0$ for $q = p + e_i - e_j$ for $i \neq j$. We avoid logarithms of zero since $p_k > 1$ and we are away from the simplex boundary. In other words, we get one inequality from every point $q$ reachable from $p$ by moving along a root of type $A_{n-1}$.

These $H_q(u) \geq 0$ inequalities are linear, with constant term zero. However, projecting the normal vectors of these hyperplanes along the all ones vector $\mathbf{1}$ and viewing $p$ as the origin of a new coordinate system, we obtain inequalities with nonzero constant terms. These inequalities describe the same logarithmic Voronoi polytope on the hyperplane $\sum_k u_k = d$. Dividing each inequality by the constant terms we obtain a system of inequalities which is of the form $Au \leq \mathbf{1}$, following the notation of \cite{Ziegler1995LecturesOnPolytopesTEXT}, where the rows of $A$ are exactly the vectors $v_{ij}$. By \cite[Theorem 2.11]{Ziegler1995LecturesOnPolytopesTEXT}, the dual polytope is given by the convex hull of these $v_{ij}$.
\end{proof}

\begin{lemma}\label{lemma:logarithmic-Voronoi-sufficient-inequalities}
Let $p \in \mathcal{M}_{n,d}$ with every entry $p_i > 1$. A sufficient system of inequalities defining the logarithmic Voronoi cell is given by the $2 \binom{n}{2}$ halfspaces $u \in \mathbb{R}^n$ such that $H_\delta(u) \geq 0$ for $\delta \in R := \left\{ e_i - e_j : i \neq j, \, \, i,j \in [n] \right\}$ and the affine plane $\sum u_i = d$, where
\begin{equation*}
    H_\delta(u) := \sum_{i \in [n]} u_i \log \left( \frac{p_i}{p_i + \delta_i} \right).
\end{equation*}
\end{lemma}

\begin{proof}
We prove that the $2\binom{n}{2}$ inequalities $H_\delta(u)\geq 0$ for $\delta \in R$ are sufficient. Fix $p\in\M$ with all $p_i > 1$. Let $u\in \RR^n$ such that $H_\delta(u)\geq 0$ for all $\delta\in R$. Fix some $q=p+\delta+\delta'$ where $\delta,\delta'\in R$, and assume that $\delta + \delta' \notin R$. We wish to show $H_q(u)=\sum_iu_i\log\frac{p_i}{q_i} \geq 0$. Consider several cases. First, if $\delta=\delta'=e_j-e_k$, it suffices to show that $$u_j\log\frac{p_j}{p_j+2}+u_k\log\frac{p_k}{p_k-2}\geq 0.$$ We claim that \begin{align}\label{main_ineq}
    u_j\log\frac{p_j}{p_j+2}+u_k\log\frac{p_k}{p_k-2}\geq 2u_j\log\frac{p_j}{p_j+1}+2u_k\log\frac{p_k}{p_k-1},
\end{align} which would be sufficient, since the right-hand side of the above equation is $\geq 0$ by assumption. We show that 
\begin{align}\label{ineqs}
    u_j\log\frac{p_j}{p_j+2}\geq 2u_j\log\frac{p_j}{p_j+1} \hspace{0.5cm}\text{and}\hspace{0.5cm} u_k\log\frac{p_k}{p_k-2}\geq 2u_k\log\frac{p_k}{p_k-1}.
\end{align}
Observe:
\begin{align*}
   &u_j\log\frac{p_j}{p_j+2}\geq 2u_j\log\frac{p_j}{p_j+1} \iff p_j^2+2p_j+1\geq p_j^2+2p_j,\\
   &u_k\log\frac{p_k}{p_k-2}\geq 2u_k\log\frac{p_k}{p_k-1}\iff p_k^2-2p_k+1\geq p_k^2-2p_k.
\end{align*}

Thus (\ref{ineqs}) holds, and we conclude that (\ref{main_ineq}) is true in this case, as desired. If $\delta\neq\delta'$, but they share both indices, then $p=q$, and we're done. If they do not share any indices, then we have that $H_q(u)=H_\delta(u)+{H_{\delta'}}(u)\geq 0$ by assumption. Suppose $\delta\neq\delta'$, and $\delta$ and $\delta'$ share one index, $j$. If $\delta=e_i-e_j$ and $\delta'=e_j-e_k$ for $i\neq j\neq k$, then $\delta+\delta'=e_i-e_k$, a contradiction to the assumption $\delta + \delta' \notin R$. Similarly when $\delta=e_j-e_i$ and $\delta'=e_k-e_j$. Suppose then that $\delta=e_i-e_j$ and $\delta'=e_i-e_k$. We wish to show that 
$$u_i\log \frac{p_i}{p_i+2}+u_j\log \frac{p_j}{p_j-1}+u_k\log\frac{p_k}{p_k-1}\geq 0.$$
Note then that
\begin{align*}
    u_i\log \frac{p_i}{p_i+2}\geq 2u_i\log\frac{p_i}{p_i+1}
    \iff & p_i^2+2p_i+1\geq p_i^2+2p_i,
\end{align*}
and the last inequality always holds for positive $p_i$, so the lemma is true for this case. The case when  $\delta=e_j-e_i$ and $\delta'=e_k-e_i$ is proved similarly. Since $H_q(u)\geq 0$ in all of the cases we considered, and the cases are exhaustive, we conclude that the lemma holds.
\end{proof}

\textbf{A family of polytopes.} For $n=2,3,4,5,6,7$ we write below the $f$-vectors for the logarithmic Voronoi cells of any point $p \in \mathcal{M}_{n,d}$ with $p_i > 1$ in all coordinates. These were computed numerically and using the face characterization of Theorem \ref{theorem:log-root-polytope-face-bijection}. The logarithmic Voronoi cells for every $\mathcal{M}_{n,d}$ are combinatorially isomorphic to the dual of the corresponding root polytope, exactly as in the Euclidean case.
\begin{equation*}
    \begin{array}{cl}
        n=2 & (1, 2, 1) \\
        n=3 & (1, 6, 6, 1) \\
        n=4 & (1, 14, 24, 12, 1)\\
        n=5 & (1, 30, 70, 60, 20, 1)\\
        n=6 & (1, 62, 180, 210, 120, 30, 1)\\
        n=7 & (1, 126, 434, 630, 490, 210, 42, 1)
    \end{array}
\end{equation*}
We have a family of Euclidean Voronoi polytopes that tile $\mathbb{R}^{n-1}$ and a family of logarithmic Voronoi polytopes that tile the open simplex $\Delta_{n-1}$. This family begins
\begin{equation*}
    \begin{array}{cccc}
        n-1=1 & n-1=2 & n-1=3 & \cdots \\
        \text{line segment} & \text{hexagon} & \text{rhombic dodecahedron} & \cdots
    \end{array}
\end{equation*}

Root polytopes of type $A$ have connections to tropical geometry. The rhombic dodecahedron is a polytrope which has been called the $3$-pyrope because of the mineral $\text{Mg}_3\text{Al}_2(\text{SiO}_4)_3$ whose pure crystal can take the same shape. For more on root polytopes, tropical geometry, and polytropes, see \cite{JoswigKulasTropicalOrdinaryConvexityCombined}.

Georgy Voronoi devoted many years of his life to studying properties of 3-dimensional parallelohedra, convex polyhedra that tessellate 3-dimensional Euclidean space. His paper on the subject called \textit{{R}echerches sur les parall\'{e}llo\`edres primitifs} \cite{VoronoiFrenchPaper} was a result of his twelve-year work. In a cover letter to the manuscript, he wrote: ``I noticed already
long ago that the task of dividing the $n$-dimensional analytical space into convex
congruent polyhedra is closely related to the arithmetic theory of positive quadratic
forms'' \cite{VoronoiBio}. Indeed, Voronoi was interested in studying cells of lattices in $\ZZ^n$ with the aim of applying them to the theory of quadratic forms.  This motivated us to study a lattice intersected with the probability simplex, the topic of our current section. Today, Voronoi decomposition finds applications to the analysis of spatially distributed data in many fields of science, including mathematics, physics, biology, archaeology, and even cinematography. In \cite{NicoZavallos2019OptimizingTheSacrifice}, the author uses Voronoi cells to optimize search paths in an attempt to improve the final 6-minute scene of Andrei Tarkovsky's \textit{Offret (the Sacrifice)}. Voronoi diagrams are so versatile they even found their way into baking: Ukrainian pastry chef Dinara Kasko uses Voronoi diagrams to 3D-print silicone molds which she then uses to make cakes \cite{DinaraKaskoCakes}.

\textbf{Acknowledgements:} Both authors would like to thank Bernd Sturmfels for suggesting this topic during the Summer of 2019, including many helpful suggestions along the way. We also thank the Max Planck Institute of Mathematics in the Sciences for support during the summers of 2019 and 2020, and also the library staff for their exceptional support during the difficult pandemic, allowing us access to the resources we need for research. The first author was also supported by the Berkeley Chancellor's fellowship.

\printbibliography

\end{document}